\documentclass[12pt]{amsart}
\usepackage{amssymb}
\usepackage{mathrsfs}
\usepackage{amsxtra}
\usepackage{graphics}
\usepackage{latexsym}
\usepackage{xcolor}
\usepackage{diagrams}
\usepackage{amsmath}
\usepackage{amssymb,amsthm,amsfonts}
\usepackage{mathtools}
\usepackage{amscd}
\usepackage[arrow, matrix, curve]{xy}
\usepackage{syntonly}
\title[]{Nonabelian Hodge theory in positive characterstic via exponential twisting}
\author[Guitang Lan]{Guitang Lan}
\email{lan@uni-mainz.de}
\address{Institut f\"{u}r  Mathematik, Universit\"{a}t
Mainz, Mainz, 55099, Germany}

\author[Mao Sheng]{Mao Sheng}
\email{msheng@ustc.edu.cn}
\address{School of Mathematical Sciences,
University of Science and Technology of China, Hefei, 230026, China}
\author[Kang Zuo]{Kang Zuo}
\email{zuok@uni-mainz.de}
\address{Institut f\"{u}r Mathematik, Universit\"{a}t
Mainz, Mainz, 55099, Germany}

\begin{document}
\theoremstyle{plain}
\newtheorem{thm}{Theorem}[section]
\newtheorem{theorem}[thm]{Theorem}
\newtheorem{lemma}[thm]{Lemma}
\newtheorem{corollary}[thm]{Corollary}
\newtheorem{proposition}[thm]{Proposition}
\newtheorem{addendum}[thm]{Addendum}
\newtheorem{variant}[thm]{Variant}
\theoremstyle{definition}
\newtheorem{lemma and definition}[thm]{Lemma and Definition}
\newtheorem{construction}[thm]{Construction}
\newtheorem{notations}[thm]{Notations}
\newtheorem{question}[thm]{Question}
\newtheorem{problem}[thm]{Problem}
\newtheorem{remark}[thm]{Remark}
\newtheorem{remarks}[thm]{Remarks}
\newtheorem{statement}[thm]{Statement}
\newtheorem{definition}[thm]{Definition}
\newtheorem{claim}[thm]{Claim}
\newtheorem{assumption}[thm]{Assumption}
\newtheorem{assumptions}[thm]{Assumptions}
\newtheorem{properties}[thm]{Properties}
\newtheorem{example}[thm]{Example}
\newtheorem{conjecture}[thm]{Conjecture}
\newtheorem{proposition and definition}[thm]{Proposition and Definition}
\numberwithin{equation}{thm}
\newcommand{\Spec}{\mathrm{Spec}}
\newcommand{\pP}{{\mathfrak p}}
\newcommand{\sA}{{\mathcal A}}
\newcommand{\sB}{{\mathcal B}}
\newcommand{\sC}{{\mathcal C}}
\newcommand{\sD}{{\mathcal D}}
\newcommand{\sE}{{\mathcal E}}
\newcommand{\sF}{{\mathcal F}}
\newcommand{\sG}{{\mathcal G}}
\newcommand{\sH}{{\mathcal H}}
\newcommand{\sI}{{\mathcal I}}
\newcommand{\sJ}{{\mathcal J}}
\newcommand{\sK}{{\mathcal K}}
\newcommand{\sL}{{\mathcal L}}
\newcommand{\sM}{{\mathcal M}}
\newcommand{\sN}{{\mathcal N}}
\newcommand{\sO}{{\mathcal O}}
\newcommand{\sP}{{\mathcal P}}
\newcommand{\sQ}{{\mathcal Q}}
\newcommand{\sR}{{\mathcal R}}
\newcommand{\sS}{{\mathcal S}}
\newcommand{\sT}{{\mathcal T}}
\newcommand{\sU}{{\mathcal U}}
\newcommand{\sV}{{\mathcal V}}
\newcommand{\sW}{{\mathcal W}}
\newcommand{\sX}{{\mathcal X}}
\newcommand{\sY}{{\mathcal Y}}
\newcommand{\sZ}{{\mathcal Z}}
\newcommand{\A}{{\mathbb A}}
\newcommand{\B}{{\mathbb B}}
\newcommand{\C}{{\mathbb C}}
\newcommand{\D}{{\mathbb D}}
\newcommand{\E}{{\mathbb E}}
\newcommand{\F}{{\mathbb F}}
\newcommand{\G}{{\mathbb G}}
\renewcommand{\H}{{\mathbb H}}
\newcommand{\I}{{\mathbb I}}
\newcommand{\J}{{\mathbb J}}
\renewcommand{\L}{{\mathbb L}}
\newcommand{\M}{{\mathbb M}}
\newcommand{\N}{{\mathbb N}}
\renewcommand{\P}{{\mathbb P}}
\newcommand{\Q}{{\mathbb Q}}
\newcommand{\Qbar}{\overline{\Q}}
\newcommand{\R}{{\mathbb R}}
\newcommand{\SSS}{{\mathbb S}}
\newcommand{\T}{{\mathbb T}}
\newcommand{\U}{{\mathbb U}}
\newcommand{\V}{{\mathbb V}}
\newcommand{\W}{{\mathbb W}}
\newcommand{\Z}{{\mathbb Z}}
\newcommand{\g}{{\gamma}}
\newcommand{\id}{{\rm id}}
\newcommand{\rk}{{\rm rank}}
\newcommand{\END}{{\mathbb E}{\rm nd}}
\newcommand{\End}{{\rm End}}
\newcommand{\Hom}{{\rm Hom}}
\newcommand{\Hg}{{\rm Hg}}
\newcommand{\tr}{{\rm tr}}
\newcommand{\Sl}{{\rm Sl}}
\newcommand{\Gl}{{\rm Gl}}
\newcommand{\Cor}{{\rm Cor}}

\newcommand{\SO}{{\rm SO}}
\newcommand{\OO}{{\rm O}}
\newcommand{\SP}{{\rm SP}}
\newcommand{\Sp}{{\rm Sp}}
\newcommand{\UU}{{\rm U}}
\newcommand{\SU}{{\rm SU}}
\newcommand{\SL}{{\rm SL}}

\newcommand{\ra}{\rightarrow}
\newcommand{\xra}{\xrightarrow}
\newcommand{\la}{\leftarrow}
\newcommand{\Nm}{\mathrm{Nm}}
\newcommand{\Gal}{\mathrm{Gal}}
\newcommand{\Res}{\mathrm{Res}}
\newcommand{\GL}{\mathrm{GL}}

\newcommand{\GSp}{\mathrm{GSp}}
\newcommand{\Tr}{\mathrm{Tr}}

\newcommand{\bA}{\mathbf{A}}
\newcommand{\bK}{\mathbf{K}}
\newcommand{\bM}{\mathbf{M}} 
\newcommand{\bP}{\mathbf{P}}
\newcommand{\bC}{\mathbf{C}}
\newcommand{\HIG}{\mathrm{HIG}}
\newcommand{\MIC}{\mathrm{MIC}}
\newcommand{\Aut}{\mathrm{Aut}}
\thanks{This work is supported by the SFB/TR 45 `Periods, Moduli
Spaces and Arithmetic of Algebraic Varieties' of the DFG, and also by the University of Science and Technology of China.}
\begin{abstract}
Let $k$ be a perfect field of odd characteristic and $X$ a smooth algebraic variety over $k$ which is $W_2$-liftable. We show that the exponent twisiting of the classical Cartier descent gives an equivalence of categories between the category of nilpotent Higgs sheaves of exponent $\leq p-1$ over $X/k$ and the category of nilpotent flat sheaves of exponent $\leq p-1$ over $X/k$, and it is equivalent up to sign to the inverse Cartier and Cartier transforms for these nilpotent objects constructed in the nonabelian Hodge theory in positive characteristic by Ogus-Vologodsky \cite{OV}. In view of the crucial role that Deligne-Illusie's lemma has ever  played in their algebraic proof of $E_1$ degeneration and Kodaira vanishing theorem in abelian Hodge theory, it may not  be overly surprising that again this lemma plays a significant role via the concept of Higgs-de Rham flow \cite{LSZ2013} in establishing $p$-adic Simpson correspondence in nonabelian Hodge theory and Langer's algebraic proof of Bogomolov inequality for semistable Higgs bundles and Miyaoka-Yau inequality \cite{Langer}.
\end{abstract}

\maketitle

\section{Introduction}
Let $k$ be a perfect field with positive characteristic and $X$ a smooth algebraic variety over $k$. We have the commutative diagram of Frobenii
 \begin{diagram}
  X&\rTo^{F_{X/k}} &X'    &\rTo^{\pi}    &X\\
     &\rdTo  &\dTo    &              &\dTo\\
     &       &{\rm Spec}\ k&\rTo^{\sigma} &{\rm Spec}\ k.\\
 \end{diagram}
In the above diagram, the composite of the relative Frobenius $F_{X/k}$ with $\pi$ is the absolute Frobenius  map $F_X$ of $X$, and $\sigma$ is the absolute Frobenius of ${\rm Spec}\ k$. A Higgs sheaf over $X/k$ is a pair $(E,\theta)$ where $E$ is a coherent sheaf of $\sO_X$-modules and $\theta: E\to E\otimes \Omega_{X/k}$ an $\sO_X$-linear morphism satisfying the integrability condition $\theta\wedge \theta=0$. It is said to be nilpotent of exponent $\leq n$ if for all local sections $\partial_1,\cdots,\partial_n$ of $T_{X/k}$,
$$
\theta(\partial_1)\cdots\theta(\partial_n)=0.
$$
Let $\HIG$ be the category of Higgs sheaves over $X/k$ and $\HIG_{\leq n}$ the full subcategory of nilpotent Higgs sheaves of exponent $\leq n$. Note that $\HIG_1$ is just the category of coherent sheaves of $\sO_X$-modules. On the other hand, we introduce the category $\MIC$ of flat sheaves over $X/k$. A flat sheaf over $X/k$ is a pair $(H,\nabla)$ with $H$ a coherent sheaf of $\sO_X$-modules together with an integrable $k$-connection $\nabla: H\to H\otimes \Omega_{X/k}$. To each $(H,\nabla)\in \MIC$ one associates the $p$-curvature map $\psi: H\to H\otimes F_X^*\Omega_{X/k}$, which is $\sO_X$-linear and satisfies $\psi\wedge\psi=0$ (see \S5.0.9 \cite{KA}). Following Definition 5.6 \cite{KA}, $(H,\nabla)$ is said to be nilpotent of exponent $\leq n$ if for all local sections $\partial_1,\cdots,\partial_n$ of $T_{X/k}$,
$$
\psi(\partial_1)\cdots\psi(\partial_n)=0.
$$
Let $\MIC_{n}$ be the full subcategory of nilpotent flat sheaves over $X/k$ of exponent $\leq n$. Note $\MIC_1$ is the category of flat sheaves with vanishing $p$-curvature. The classical Cartier descent theorem is as follows:
\begin{theorem}[Theorem 5.1 \cite{KA}]\label{Cartier}
There is an equivalence of categories between $\HIG_1$ and $\MIC_1$. Explicitly, one associates $(F_X^*E,\nabla_{can})$ to $(E,0)\in \HIG_1$ and conversely, one associates $\pi_{*}H^{\nabla}$ to $(H,\nabla)\in \MIC_1$.
\end{theorem}
In the above theorem, $\nabla_{can}$ means the unique connection on $F_{X}^*E$ such that the pullback of any local section of $E$ is flat, and the $k$-subsheaf $H^{\nabla}$ of flat sections is naturally a $\sO_{X'}$-module of the same rank as $H$. By abuse of notations, we omit $\pi_*$ by assuming this identification of an object over $X'$ with the corresponding object over $X$. \\

In the recent spectacular work \cite{OV}, Ogus and Vologodsky have established the nonabelian Hodge theory in positive characteristic. Among other important results, they have generalized Theorem \ref{Cartier} in a far-reaching way, and also the fundamental $p$-curvature formula of the Gauss-Manin connection of Katz \cite{KA71} and the fundamental decomposition theorem of Deligne-Illusie \cite{DI}. A special but essential case of their main construction in loc. cit. is the following
\begin{theorem}[Thereom 2.8 \cite{OV}]
Suppose $X$ is $W_2$-liftable. Then there is an equivalence of categories
$$
\HIG_{p-1}\xrightleftharpoons[\ C\ ]{\ C^{-1}\ }\MIC_{p-1}.
$$
\end{theorem}
\begin{remark}
A $W_2$-lifting $\tilde X$ of $X$ induces the $W_2$-lifting $\tilde{X}':= \tilde X\times_{\Spec\ W_2,\sigma} \Spec\ W_2$ of $X'$. Put $(\sX,\sS)=(X/k,\tilde X'/W_2)$. Then the above functor $C^{-1}$ is given by $C^{-1}_{\sX/\sS}\circ \pi^*$ and $C$ given by $\pi_*C_{\sX/\sS}$, where $C^{-1}_{\sX/\sS}$ (resp. $C_{\sX/\sS}$) is the inverse Cartier transform (resp. Cartier transform) with respect to the pair $(\sX,\sS)$ in \cite{OV}, restricting to the above subcategories. The full categories in loc. cit. have the merit of being tensor categories and their functors are compatible with tensor product.
\end{remark}
The present note is aimed towards readers who wish to understand their fundamental functors (essentially $C^{-1}$ and $C$) in a more explicit (but therefore less elegant) way. Our point of view is that the (inverse) Cartier transform is nothing but the exponential twisting of the Cartier descent theorem. For a technical reason (see \S2), we have to assume ${\rm char}(k)=p$ to be odd from now on. In \S2, we construct a functor $C^{-1}_{\exp}$ from $\HIG_{p-1}$ to $\MIC_{p-1}$ and then a functor $C_{\exp}$ in the converse direction. Then in \S3, we show that there is a natural isomorphism between the functor $C^{-1}_{\exp}$ (resp. $C_{\exp}$) and the functor $C^{-1}$ (resp. $C$). In \S4, we explain that the Gauss-Manin flat bundle of a Fontaine-Faltings module can be reconstructed from the associated graded Higgs bundle via the inverse Cartier transform.

\section{Exponential twisting}

\subsection{Deligne-Illusie's Lemma}
Rewrite $X$ by $X_0$. Choose and then fix a $W_2$-lifting $X_1$ of $X_0$. Then take an affine covering $\mathcal{U}=\{ \tilde{U}_{\alpha}\}_{\alpha\in I}$ of $X_1$ and for each $\tilde{U}_{\alpha}$, take a Frobenius lifting $\tilde{F}_{\alpha}: \tilde{U}_{\alpha}\to \tilde{U}_{\alpha}$ which mod $p$ is the absolute Frobenius $F_0: U_{\alpha}\to U_{\alpha}$. Here $U_{\alpha}$ means the closed fiber of $\tilde{U}_{\alpha}$. In the following, we shall always use $F_0$ for the absolute Frobenius on any variety over $k$. The induced morphism by $\tilde{F}_{\alpha}$ on differential forms over $\tilde{U}_{\alpha}$ is therefore divisible by $p$. The composite of $\sO_{\tilde{U}_{\alpha}}$-morphisms
$$
\tilde{F}_{\alpha}^*\Omega_{\tilde{U}_{\alpha}} \stackrel{\tilde{F}_{\alpha}}{\rightarrow} p\Omega_{\tilde{U}_{\alpha}}\stackrel{\frac{1}{[p]}}{\cong} \Omega_{U_{\alpha}}
$$
induces an $\sO_{U_{\alpha}}$-morphism
$$\zeta_{\alpha}:=\frac{\tilde{F}_{\alpha}}{[p]}: F_0^{*}\Omega_{U_{\alpha}}\to \Omega_{U_{\alpha}}.$$
The basic lemma of Deligne-Illusie in \cite{DI} is the following:
\begin{lemma}\label{lemma D-I} There are homomorphisms $h_{\alpha \beta}:F_0^*\Omega_{U_{\alpha\beta}}\rightarrow \mathcal{O}_{U_{\alpha\beta}}$, satisfying the following two properties:
\begin{itemize}
	\item [(i)]over $F_0^{-1}\Omega_{U_{\alpha\beta}}$ we have $\zeta_{\alpha}-\zeta_{\beta}=dh_{\alpha\beta};$
	\item [(ii)] the cocycle condition over $U_{\alpha\beta\gamma}$: $h_{\alpha\beta} + h_{\beta\gamma}=h_{\alpha\gamma}$.
\end{itemize}
\end{lemma}
 \begin{proof}
Consider the $W_2$-morphism $G_\alpha: Z'_{\alpha}\to U'_{\alpha\beta}:=U'_{\alpha}\cap U'_{\beta}$ sitting in the following Cartesian diagram:
\[
\begin{CD}
 Z'_{\alpha}@> G_{\alpha}>>  U'_{\alpha\beta} \\
 @Vj'_{\alpha}VV       @VVi'_{\alpha}V\\
 U'_{\alpha} @>F_{\alpha}>> U'_{\alpha}
 \end{CD}
\]
where $i'_{\alpha}$ is the natural inclusion. By reduction modulo $p$, we obtain the following Cartesian square
\[
\begin{CD}
  Z_{\alpha}@>G_0 >>U_{\alpha\beta}     \\
 @Vj_{\alpha}VV       @VVi_{\alpha}V\\
 U_{\alpha} @>F_0>> U_{\alpha}.
 \end{CD}
\]
Thus we see that $Z_{\alpha}$ is $U_{\alpha\beta}$ and $G_{\alpha}: Z'_{\alpha}\to U'_{\alpha}$ is a lifting of the absolute Frobenius $F_0$ over $U_{\alpha\beta}$. Similarly for $(U'_{\beta},F_{\beta})$, we have  $G_{\beta}:Z'_{\beta}\rightarrow U'_{\alpha\beta} $ which is also a lifting of $F_0:U_{\alpha\beta}\rightarrow U_{\alpha\beta}$. Now we apply Lemma 5.4 \cite{IL} to the pair $(G_{\alpha}:Z'_{\alpha}\rightarrow  U'_{\alpha\beta}, \ G_{\beta}:Z'_{\beta}\rightarrow  U'_{\alpha\beta})$ of Frobenis liftings of the absolute Frobenius $F_0$ on $U_{\alpha\beta}$, we get the homomorphisms $h_{\alpha \beta}:F_0^*\Omega_{U_{\alpha\beta}}\rightarrow \mathcal{O}_{U_{\alpha\beta}}$ such that over $F_0^{-1}\Omega_{U_{\alpha\beta}}$ we have $\zeta_{\alpha}-\zeta_{\beta}=dh_{\alpha\beta}$ and  $h_{\alpha\beta}+h_{\beta\gamma}=h_{\alpha\gamma}$.
\end{proof}

\subsection{Inverse Cartier} \label{inverse cartier} Given a Higgs sheaf $(E,\theta)\in \HIG_{p-1}$, we are going to associate to it a flat sheaf $C_{\exp}^{-1}(E,\theta)\in \MIC_{p-1}$. \\

{\itshape Description:} Over each $U_{\alpha}$, we define a local sheaf $H_{\alpha}:=F_0^*(E|_{U_{\alpha}})$ together with a connection over $H_{\alpha}$ by the formula
$$
\nabla_{\alpha}=\nabla_{can}+\zeta_{\alpha}(F_0^*\theta|_{U_{\alpha}})
$$
where $E|_{U_{\alpha}}$ is the restriction of $E$ to $U_{\alpha}$ (similar for the meaning of $\theta|_{U_{\alpha}}$) and $\nabla_{can}$ is the flat connection in the Cartier descent theorem \ref{Cartier}. Over $U_{\alpha\beta}:=U_{\alpha}\cap U_{\beta}$, after Lemma \ref{lemma D-I}, we define
an $\mathcal{O}_{U_{\alpha\beta}}$-linear morphism
$$
h_{\alpha\beta}(F_0^*\theta): F_0^*E|_{U_{\alpha\beta}}\to F_0^*E|_{U_{\alpha\beta}}.
$$
Because $\theta$ is by assumption nilpotent of exponent $\leq p-1$, we are able to define an element in $G_{\alpha\beta}\in \Aut_{\sO_{U_{\alpha\beta}}}(F_0^*E|_{U_{\alpha\beta}})$ by the formula
$$
\exp[h_{\alpha\beta}(F_0^*\theta)]:=\sum_{i=0}^{p-1}\frac{(h_{\alpha\beta}(F_0^*\theta))^i}{i!}.
$$
Then we use the set of local isomorphisms $\{G_{\alpha\beta}\}_{\alpha,\beta\in I}$ to glue $\{(H_{\alpha},\nabla_{\alpha})\}_{\alpha\in I}$, to obtain the claimed flat sheaf $C_{\exp}^{-1}(E,\theta)$ over $X/k$. The verification details are contained in the following steps:\\

{\itshape Step 1: Local sheaves glue.}
It is to show the cocycle condition: $$G_{\beta\gamma}\circ G_{\alpha\beta}=G_{\alpha\gamma}.$$
We compute
\begin{eqnarray*}
 G_{\beta\gamma}\circ G_{\alpha\beta}&=&\exp[h_{\beta\gamma}(F_0^*\theta)]\exp[h_{\alpha\beta }(F_0^*\theta)].\\
  \end{eqnarray*}
It follows from the integrability of Higgs field that the two morphism $h_{\alpha\beta}(F_0^*\theta)$ and $h_{\beta\gamma}(F_0^*\theta)$ commute with each other. Thus we compute further that
\begin{eqnarray*}
  G_{\beta\gamma}\circ G_{\alpha\beta}=\exp[(h_{\beta\gamma}+h_{\alpha\beta})(F_0^*\theta)]=\exp[h_{\alpha\gamma}(F_0^*\theta)]=G_{\alpha\gamma}.
\end{eqnarray*}
The second equality follows from Lemma \ref{lemma D-I} (ii).\\

{\itshape Step 2: Local connections glue.}
It is to show that the local connections $\{\nabla_{\alpha}\}$ coincide on the overlaps, that is
$$
 (G_{\alpha\beta}\otimes id)\circ \nabla_{\alpha}= \nabla_{\beta}\circ G_{\alpha\beta}.
$$
It suffices to show
$$
 \zeta_{\alpha}(F_0^*\theta)= G^{-1}_{\alpha\beta}\circ dG_{\alpha\beta}+ G_{\alpha\beta}^{-1} \circ\zeta_{\beta}(F_0^*\theta)\circ G_{\alpha\beta}.
$$
We see that
$$
 G^{-1}_{\alpha\beta}\circ dG_{\alpha\beta}=-dG^{-1}_{\alpha\beta}\circ G_{\alpha\beta}=dh_{\alpha\beta}(F_0^*\theta),
$$
and
$$
G_{\alpha\beta}^{-1}\circ \zeta_{\beta}(F_0^*\theta)\circ G_{\alpha\beta}= \zeta_{\beta}(F_0^*\theta),
$$
as $G_{\alpha \beta}$ commutes with $\zeta_{\beta}(F_0^*\theta)$ due to the integrability of the Higgs field.
So
$$
G^{-1}_{\alpha\beta}\circ dG_{\alpha\beta}+ G_{\alpha\beta}^{-1} \circ\zeta_{\beta}(F_0^*\theta)\circ G_{\alpha\beta}=dh_{\alpha\beta}(F_0^*\theta)+\zeta_{\beta}(F_0^*\theta)=\zeta_{\alpha}(F_0^*\theta).
$$
The last equality uses Lemma \ref{lemma D-I} (i). \\

{\itshape Step 3: Flatness.} This is a local property. First of all, one has
$$F_0^*(\theta)\wedge F_0^*(\theta)=F_0^*(\theta\wedge \theta)=0,$$ and then
$$
\zeta_{\alpha}(F_0^*\theta)\wedge  \zeta_{\alpha}(F_0^*\theta)=(\bigwedge^2\zeta_{\alpha})(F_0^*\theta\wedge F_0^*\theta)=0.
$$
(Here is the place where we need to exclude $p=2$ as we use the second wedge product.) It is left to show that $d(\zeta_{\alpha}(F_0^*\theta))=0$.
This is done by a local computation: by definition, for $\omega\in \Omega_{U_{\alpha}}$,
$$\zeta_{\alpha}(F_0^*\omega)= \frac{1}{[p]}(dF_{\alpha} (F_{\alpha}^*\omega')),$$
where $\omega'\in \Omega_{U'_{\alpha}}$ is any lifting of $\omega$.
Then
$$d\circ\zeta_{\alpha}(F_0^*\omega)=d\circ \frac{1}{[p]}(dF_{\alpha} (F_{\alpha}^*\omega'))=\frac{1}{[p]}(d\circ dF_{\alpha} (F_{\alpha}^*\omega')).$$
We may write $\omega'=\sum_i f_idg_i$ for $f_i,g_i\in \mathcal{O}_{U'_{\alpha}}$. Then
$$d(dF_{\alpha}(F_{\alpha}^*\omega'))=\sum_i d(F_{\alpha}^*f_i)\wedge d(F_{\alpha}^*g_i)\in p^2\Omega^2_{U'_{\alpha}}=0.$$ Thus $d(\zeta_{\alpha}(F_0^*\omega))=0$. Clearly it follows that $d(\zeta_{\alpha}(F_0^*\theta))=0$. \\

{\itshape Step 4: Nilpotency.} The $p$-curvature of the flat sheaf $C_{\exp}^{-1}(E,\theta)$ over $U_{\alpha}$ takes the form
$$
F_0^*\theta_{U_{\alpha}}: F_{0}^*E|_{U_{\alpha}}\to F_{0}^*E|_{U_{\alpha}}\otimes F_0^*\Omega_{U_{\alpha}}
$$
which is clearly nilpotent of exponent $\leq p-1$.

\subsection{Cartier}\label{cartier transform} We shall do the converse process, namely,  associate a Higgs sheaf $C_{\exp}(H,\nabla)\in \HIG_{p-1}$ to any flat sheaf $(H,\nabla)\in \MIC_{p-1}$. \\

{\itshape Description:} Let $\psi: H\to H\otimes F_0^{*}\Omega_{X/k}$ be the $p$-curvature map of $(H,\nabla)$. Set
$$
H_{\alpha}:=H|_{U_{\alpha}}, \quad \nabla_{\alpha}:=\nabla|_{U_{\alpha}}, \quad \psi_{\alpha}:=\psi|_{U_{\alpha}}.
$$
Define a new connection $\nabla'_{\alpha}$ on $H_{\alpha}$ by the formula:
$$
\nabla_{\alpha}'=\nabla_{\alpha}+ \zeta_{\alpha}(\psi_{\alpha}).
$$
Because of the nilpotency condition on $\psi$, we can use again Lemma \ref{lemma D-I} to define
$$
J_{\alpha\beta}:=\exp(h_{\alpha\beta}(\psi))\in \Aut_{\sO_{U_{\alpha\beta}}}(H|_{U_{\alpha\beta}}).
$$
Then we use the set $\{J_{\alpha\beta}\}_{\alpha,\beta\in I}$ of local isomorphisms to glue $\{H_{\alpha},\nabla_{\alpha}'\}_{\alpha\in I}$, to obtain a new flat sheaf $(H',\nabla')$ whose $p$-curvature vanishes. The $p$-curvature map $\psi$ induces an $F$-Higgs sheaf in a natural way
$$
\psi': H'\to H'\otimes F_0^*\Omega_{X/k},
$$
which is parallel with respect to $\nabla'$. By the Cartier descent theorem \ref{Cartier}, the pair $(H',\psi')$ descend to a Higgs sheaf $C_{\exp}(H,\nabla)$. The verification steps are given as follows:\\

{\itshape Step 1: Local sheaves glue.}  This step follows from Lemma \ref{lemma D-I} (ii) in a similar way to Step 1 of inverse Cartier.  Also, it is direct to check that
$$
 \psi_{\beta}\circ  J_{\alpha\beta}=J_{\alpha\beta} \circ \psi_{\alpha}.
$$
Hence $\psi$ induces an $F$-Higgs sheaf $\psi': H'\to H'\otimes F_0^*\Omega_{X/k}$.   \\

{\itshape Step 2: Local connections glue.} It is to show
 $$\zeta_{\alpha}(\psi)=J^{-1}_{\alpha\beta}\circ dJ_{\alpha\beta}+ J_{\alpha\beta}^{-1}\circ \zeta_{\beta}(\psi)\circ J_{\alpha\beta}.$$
As by Lemma \ref{lemma D-I} (i)
 $$
 J^{-1}_{\alpha\beta}\circ dJ_{\alpha\beta}=-dJ_{\alpha\beta}^{-1}\circ J_{\alpha\beta}= d( \psi(h_{\alpha\beta}))= \zeta_{\alpha}(\psi)-\zeta_{\beta}(\psi),
 $$
 and
 $$
 J_{\alpha\beta}^{-1}\circ\zeta_{\beta}(\psi)\circ J_{\alpha\beta}= \zeta_{\beta}(\psi),
 $$
it follows that
 $$
 J^{-1}_{\alpha\beta}\circ dJ_{\alpha\beta}+ J_{\alpha\beta}^{-1}\circ \zeta_{\beta}(\psi)\circ J_{\alpha\beta}=\zeta_{\alpha}(\psi)-\zeta_{\beta}(\psi) +\zeta_{\beta}(\psi)=\zeta_{\alpha}(\psi).
 $$

{\itshape Step 3: $P$-curvature of the new connection vanishes.}  This is a local check, which is reduced to check an elementary polynomial identity in char $p$ (see Statement \ref{statement} below). The following proof relies on instead a trick due to S. Mochizuki. Take a system of \'{e}tale local coordinates $\{t_1,\cdots,t_d\}$ for $U_{\alpha}$. Write $\sO$ for $\sO_{U_{\alpha}}$ and $\Omega$ for $\Omega_{U_{\alpha}}$ and so on. Consider an auxiliary sheaf $\sO\oplus F_0^*\Omega$, equipped with the connection $\nabla_1$ defined by the formula:
 $$
 (f,g\otimes\omega) \mapsto (df+g\otimes \zeta_{}\left(1\otimes \omega),(1\otimes \omega)\otimes dg\right), \quad f,g\in \sO,\ \omega\in \Omega.
$$
Its $p$-curvature $\psi_{\nabla_1}$ takes a simple form. In fact, it is zero over the factor $\sO$, and for $\omega$,
$$
\psi_{\nabla_1}(1\otimes \partial t_i)(1\otimes dt_j)=(\nabla_{1}(1\otimes \partial t_i))^p(1\otimes dt_j)=(\frac{\partial}{\partial t_i})^{p-1}(\partial t_i\lrcorner \zeta_{\alpha}(1\otimes dt_j))=-\delta_{ij},
$$
where $\delta_{ij}$ is the Dirichlet delta-function. So we obtain the formula for $\psi_{\nabla_1}$:
$$
\psi_{\nabla_1}(f,g\otimes \omega)=-(g,0)\otimes \omega.
$$
Next we extend the connection $\nabla_1$ to the connection $\nabla_2$ on
$$
F_0^*(S^\cdot \Omega)=\sO\oplus F_0^*\Omega\oplus F_0^*S^2\Omega\oplus \cdots
$$
by the Leibniz rule. From above, it is easy to see the formula for its $p$-curvature $\psi_{\nabla_2}$:
$$
\psi_{\nabla_2}(1\otimes \partial t_i)(1\otimes \omega)=-1\otimes \partial t_i \lrcorner \omega, \quad \omega\in S^\cdot \Omega.
$$
Finally, $\nabla_2$ induces the dual connection $\nabla_3$ on its dual $F_0^*(S^{\cdot}T)$. Set
$$f_{ij}:=\partial t_j \lrcorner \zeta_{\alpha}(1\otimes dt_i)$$, $1_{\Omega}$ the unit of $F_0^*(S^\cdot \Omega)$, $1_T$ the unit of $F_0^*(S^\cdot T)$, and $<\cdot>$ the natural pairing between $F_0^*(S^\cdot T)$ and $F_0^*(S^\cdot \Omega)$. It holds that
$$
 <\nabla_{3}(\partial t_i)(1_T),1_{\Omega}>=-<1_T,\nabla_{2}(\partial t_i)(1_{\Omega})>=0,
 $$
and
 $$
 <\nabla_{3}(\partial t_i)(1_T), 1\otimes dt_j>=-<1_T, \nabla_{2}(\partial t_i)(1\otimes dt_j>=-f_{ji},
 $$
 and for $m\geq 2$
 $$
 <\nabla_{3}(\partial t_i),1\otimes dt_{i_1}\cdots dt_{i_m}>=0.
 $$
 So we obtain the formula for $\nabla_3$:
 \begin{equation}
  \nabla_{3}(\partial t_i)(1_T)=\sum_{j=1}^d -f_{ji}\otimes \partial t_j.
 \end{equation}
It follows that for its $p$-curvature $\psi_{\nabla_3}$, one has
\begin{eqnarray*}
&&<\psi_{\nabla_3}(1\otimes \partial t_i)(1\otimes(\partial t_1)^{i_1}\cdots (\partial t_d)^{i_d}),1\otimes (dt_1)^{j_1}\cdots (dt_d)^{j_d}>\\
 &=&<[\nabla_3(\partial t_i)]^p(1\otimes(\partial t_1)^{i_1}\cdots (\partial t_d)^{i_d}),1\otimes (dt_1)^{j_1}\cdots (dt_d)^{j_d}>\\
 &=&-<1\otimes(\partial t_1)^{i_1}\cdots (\partial t_d)^{i_d}, [\nabla_2(\partial t_i)]^p(1\otimes (dt_1)^{j_1}\cdots (dt_d)^{j_d})>\\
 &=&-<1\otimes(\partial t_1)^{i_1}\cdots (\partial t_d)^{i_d},\psi_{\nabla_2}(1\otimes \partial t_i)(1\otimes (dt_1)^{j_1}\cdots (dt_d)^{j_d})>\\
 &=&<1\otimes(\partial t_1)^{i_1}\cdots (\partial t_d)^{i_d}, 1\otimes\partial t_i\lrcorner( (dt_1)^{j_1}\cdots (dt_d)^{j_d})>\\
 &=&<1\otimes \partial t_i(\partial t_1)^{i_1}\cdots (\partial t_d)^{i_d}, 1\otimes (dt_1)^{j_1}\cdots (dt_d)^{j_d}>.
\end{eqnarray*}
So for $\tau\in F_0^*(S^\cdot T)$, $ \upsilon\in F_0^*T$,
 \begin{equation}\label{multiplication}
 \psi_{\nabla_3}(\upsilon)(\tau)= \upsilon\cdot \tau.
 \end{equation}
Now we let $F_0^*T$ act on $F_0^*(S^\cdot T)$ via $\psi_{\nabla_3}$, which extends to an $L:=F_0^*(S^\cdot T)$-action on $F_0^*(S^\cdot T)$. By Formula \ref{multiplication}, this action is just the multiplication map. So $F_0^*(S^\cdot T)$ is a rank one free module over $L$ with the basis $1_T$. Recall that $L$ acts on $H$ via the $p$-curvature map $\psi$. Therefore there is an isomorphism of $\sO$-modules:
$$
 \lambda:  \mathscr{H}om_{L}(F_0^*(S^\cdot T),H)\cong H,\quad \phi\mapsto \phi(1_T).
$$
The connection $\nabla_3$ on $F_0^*(S^{\cdot}T)$ and $\nabla$ on $H$ induces the connection $\nabla_4$ on $\mathscr{H}om_{L}(F_0^*(S^\cdot T),H)$, which via the above isomorphism $\lambda$ induce the connection $\nabla''$ on $H$. Claim that $\nabla''=\nabla'$. Note that the claim completes this step, since $L$ acts on $\mathscr{H}om_{L}(F_0^*(S^\cdot T),H)$ tautologically by zero and therefore the $p$-curvature map of $\nabla''$ is simply the zero map. So, for any $\phi\in \mathscr{H}om_{L}(F_0^*(S^\cdot T),H)$ and the corresponding $e:=\phi(1_T)$, one calculates that
\begin{eqnarray*}
\nabla''(\partial t_i)(e)&=&(\nabla_{4}(\partial t_i)\phi)(1_T)\\
&=&\nabla(\partial t_i)(\phi(1_T))-\phi(\nabla_{3}(\partial t_i)(1_T))\\
&=&\nabla (\partial t_i)(e)+\sum_{j=1}^d\phi(f_{ji}\otimes\partial t_j)\\
&=& \nabla (\partial t_i)(e)+\sum_{j=1}^d\phi(\psi_{\nabla_3}(f_{ji}\otimes\partial t_i)(1_T))\\
&=&\nabla(\partial t_i)(e)+\sum_{j=1}^d\psi_{\nabla_3}(f_{ji}\otimes\partial t_j)(\phi(1_T))\\
&=&\nabla(\partial t_i)(e)+\partial t_i\lrcorner\zeta(\psi(e)).
\end{eqnarray*}
Thus $\nabla''=\nabla+\zeta(\psi)=\nabla'$ as claimed.\\

{\itshape Step 4: Nilpotency.} It is clear that the so-obtained Higgs sheaf $C_{\exp}(H,\nabla)$ is nilpotent of exponent $\leq p-1$.

\begin{remark} One could probably verify directly the following statement to avoid using Mochizuki's trick:
\begin{statement}\label{statement}
Let $n$ denote the nilpotent exponent of the connection, then $n\leq p-1$.
 For any $k\leq n$, set
 $$
   f_k(T_1,\cdots, T_k)=\sum_{0\leq a_1+\cdots a_k\leq
   p-k}(1+T_k)^{a_k}(1+T_k+T_{k-1})^{a_{k-1}}\cdots (1+T_k+\cdots
   T_1)^{a_1}
 $$
 here, the index $a_1,\cdots,a_k$ are natural numbers. And set
$$
 F_k(T_1,\cdots T_k)=\sum_{\sigma\in \mathcal{S}_k}f_k(T_{\sigma(1)},\cdots
 T_{\sigma(k)}),
$$
here $\mathcal{S}_k$ is the permutation group of $k$ elements. Then  for  $k>1$, $F_k$ mod $p$ is zero.
\end{statement}
\end{remark}
\begin{remark}
It is not difficult to see that $C^{-1}_{\exp}$ and $C_{\exp}$ (up to isomorphisms) depend neither on the choices of coverings $\sU$ nor on the choices of Frobenius liftings. However, they depend on the choice of $W_2$-liftings of $X$. Also, it is clear that they respect a $G$-structure on flat sheaves or Higgs sheaves for any subgroup $G\subset \GL$.
\end{remark}

\section{Equivalence of two constructions}
In this section we want to show that our construction coincides with Ougs-Vologodsky's abstract construction up to a minus sign. The notions in \cite{OV} will be directly applied in the following.  \\

For $(E,\theta)\in \HIG_{p-1}$, let $(E',\theta'):=\pi^*(E,\theta)$ be the pull-back Higgs sheaf over $X'_0$. Then $C_{(\sX,\sS)}^{-1}(E',\theta')$ is defined as
$$
(M,\nabla_M):=\mathscr{B}_{\mathscr{X}/\mathscr{S}}\otimes_{\hat{\Gamma}_{\cdot}T_{X'_0}}\iota^*(E').
$$
Now for  $(\tilde{U}_{\alpha}, \tilde{F}_{\alpha})$, we set $U_{\alpha}':=U_{\alpha}\times_{W_2,\sigma} W_2$ and $F'_{\alpha}$ the composite of $\tilde{F}_{\alpha}$ and $\tilde{\pi}^{-1}$. They provide an isomorphism
$$
\sigma_{\alpha}:\mathscr{B}_{\mathscr{X}/\mathscr{S}}|_{U_{\alpha}}\cong F^*\hat{\Gamma}_{\cdot} T_{U_{\alpha}'},
$$
which induces isomorphisms:
$$
 \xi_{\alpha}:M|_{U_{\alpha}}\cong  F^*\hat{\Gamma}_{\cdot}T_{U_{\alpha}'}\otimes_{\hat{\Gamma}_{\cdot}T_{U_{\alpha}'}} \iota^*E'\cong  F^*_0\hat{\Gamma}_{\cdot}T_{U_{\alpha}}\otimes_{\hat{\Gamma}_{\cdot}T_{U_{\alpha}}} \iota^*E,
$$
and

$$
\varpi_{\alpha}: F^*_0\hat{\Gamma}_{\cdot}T_{U_{\alpha}}\otimes_{\hat{\Gamma}_{\cdot}T_{U_{\alpha}}} \iota^*E\cong F_0^*\iota^*E\cong F_0^*E.
$$
Set $\eta_{\alpha}:=\varpi_{\alpha}\circ \xi_{\alpha}$. Then under the isomorphism $\eta_{\alpha}$, $\nabla_M$ induces a connection $\nabla_{\alpha}$ on $F_0^*E$. Then for any local section $e$ of  $E$ over $U_{\alpha}$,
\begin{eqnarray*}
\nabla_{\alpha}(\partial t_k)(1\otimes e)&=&\varpi_{\alpha}(\nabla_{\partial t_k}(1_T)\otimes e)\\
&=&\varpi_{\alpha}(\sum_{i=1}^d(f_{ik}\otimes\partial t_i)\otimes e)\\
&=&-\sum_{i=1}^d\varpi_{\alpha}(f_{ik}\otimes\theta_{\partial t_i}(e))\\
&=&-\sum_{i=1}^df_{ik}\varpi_{\alpha}( 1\otimes\theta_{\partial t_i}(e))\\
&=&-\partial t_k\lrcorner \zeta_{\alpha}(\theta(e)).
\end{eqnarray*}
Notice that the connection on $F_0^*\hat{\Gamma}_{\cdot}T_{U_{\alpha}}$ in \cite{OV} differs from ours by a minus sign. So
$$\nabla_{\alpha}=\nabla_{can} -\zeta(\theta).$$
Now  on the overlap $U_{\alpha\beta}$, $\sigma_{\beta}\circ\sigma_{\alpha}^{-1}$ is just the multiplication map by $\exp(h_{\alpha\beta})$, here $h_{\alpha\beta}$ is considered as a local section of $F^*T_{U_{\alpha\beta}'}$. So if we set $\bar{J}_{\alpha\beta}:= \eta_{\beta}\circ \eta_{\alpha}^{-1}$, then
$$
\bar{J}_{\alpha\beta}(1\otimes e)= \exp(h_{\alpha\beta})\otimes \pi^*(e)= \exp(-h_{\alpha\beta}(F_0^*\theta(e))).
$$
It follows that the inverse Cartier transform $C_{(\sX,\sS)}^{-1}(\pi^*(E,\theta))$ is equivalent to using
$$\{\bar{J}_{\alpha\beta}=\exp(-h_{\alpha\beta}(F_0^*\theta)) \}$$
to glue the local models
$$
\{(M_{\alpha}=F_0^*E|_{U_{\alpha}}, \nabla_{\alpha}=\nabla_{can}-\zeta_{\alpha}(\theta))\},
$$
which is just $C_{\exp}^{-1}(E,-\theta)$.\\

On the other hand, given a flat sheaf $(H,\nabla)\in \MIC_{p-1}$, its Cartier transform $C_{(\sX,\sS)}(H,\nabla)$ is defined by
$$
(E',\theta')= \iota^*\mathscr{H}om_{D^{\gamma}_{X_0/S}}(\mathscr{B}_{\mathscr{X}/\mathscr{S}},H).
$$
Let $D_{X_0/k}$ be the sheaf of PD-differential operators on $X_0$. Set $D_{\alpha}:=D_{X_0/k}|_{U_{\alpha}}$, and $(E,\theta):=\pi_*(E',\theta')$. As the $p$-curvature map $\psi$ of $(H,\nabla)$ is nilpotent of exponent $\leq p-1$, the above $\sigma_{\alpha}$ induces isomorphisms:
$$
 \mu_{\alpha}:E|_{U_{\alpha}}\cong \iota^*\mathscr{H}om_{F^*_0T_{U_{\alpha}}}(F^*_0T_{U_{\alpha}},H)^{D_{\alpha}},
$$
and
$$
  \kappa_{\alpha}:\iota^*\mathscr{H}om_{F^*_0T_{U_{\alpha}}}(F^*_0T_{U_{\alpha}},H)^{D_{\alpha}} \cong  H_{\alpha} ^{\nabla_{\alpha}'}.
$$
Here $\nabla_{\alpha}'$ is the connection on $H_{\alpha}$ induced from that on $\mathscr{H}om_{F^*_0T_{U_{\alpha}}}(F^*_0T_{U_{\alpha}},H_{\alpha})$ which is in turn induced by the connection $\nabla_{T_{\alpha}}$ on $F^*_0T_{U_{\alpha}}$ and $\nabla$ on $H$. Now the calculation in the Step 3 of our Cartier construction in \S2 shows that
$$\nabla_{\alpha}'= \nabla_H+ \zeta_{\alpha}(\psi).$$
Set $\rho_{\alpha}:=\kappa_{\alpha}\circ \mu_{\alpha}$. Then via the isomorphism $\rho_{\alpha}$, the Higgs field $\theta$ induces a Higgs field ${\theta_{\alpha}}$ on $ H^{\nabla_{\alpha}'} $. Then for any  local section $e=\phi(1_T)$ of $H$ over $U_{\alpha}$ annihilated by $\nabla_{\alpha}'$ with $\phi$ a local section of $\mathscr{H}om_{F^*_0T_{U_{\alpha}}}(F^*_0T_{U_{\alpha}},H)$ over $U_{\alpha}$, one has
$$
\theta_{\alpha}(\partial t_k)(e)= \kappa_{\alpha}(\psi(\partial t_k)\circ\phi)=-\psi(\partial t_k)(\phi(1_T))=-\psi(\partial t_k)(e).
$$
That means $\theta_{\alpha}= -\psi$. Now on the overlap $U_{\alpha\beta}$, if we set $J_{\alpha\beta}:=\rho_{\beta}\circ \rho_{\alpha}^{-1}$, then for $e\in H^{\nabla'_{\alpha}}$, with $e=\phi(1_T)$ for $\phi\in \mathscr{H}om_{F^*_0T_{U_{\alpha}}}(F^*_0T_{U_{\alpha}},H)$, one has
\begin{eqnarray*}
J_{\alpha\beta}(e)&=& \phi(\exp(h_{\beta\alpha}))\\
&=&\phi(\psi_T(\exp(h_{\alpha\beta}))(1_T))\\
&= &\exp(\psi(h_{\alpha\beta}))(\phi(1_T))\\
&=&\exp(\psi(h_{\alpha\beta}))(e),
\end{eqnarray*}
which means
$$
J_{\alpha\beta}=\exp(\psi(h_{\alpha\beta})).
$$
It follows that the Cartier transform $\pi_*C_{(\sX,\sS)}(H,\nabla)$ is naturally isomorphic to $C_{\exp}(H,-\nabla)$.

\section{Gauss-Manin connection of a Fontaine-Faltings module}
Let ${\bf X}$ be a smooth scheme over $W$. G. Faltings, generalizing the work of Fontaine-Laffaille to a geometric base, introduced the category $\mathcal{MF}^{\nabla}_{[0,n]}({\bf X}/W), n\leq p-1$ in \cite{FA}. An object in this category shall be called a Fontaine-Faltings module. Under a mild condition, for a smooth proper morphism $f:{\bf  Y}\to {\bf X}$ over $W$, the higher direct images of the constant crystal over ${\bf Y}/W$ are objects in this category. In that case, let $H$ be the hypercohomology of the relative de Rham complex of $f$, $\nabla$ the Gauss-Manin connection, $Fil$ the Hodge filtration and $\Phi$ the relative Frobenius. Then the four tuple $(H,\nabla,Fil,\Phi)$ makes a Fontaine-Faltings module coming from geometry. In general, a Fontaine-Faltings module may not come from geometry. We intend to point out a relation of the mod $p$ reduction of $(H,\nabla)$ with the above theory. We assume from now on that $(H,\nabla,Fil,\Phi)$ is a strict $p$-torsion Fontaine-Faltings module, i.e. $pH=0$ (one considers otherwise its mod $p$-reduction). Let $(E,\theta)=Gr_{Fil}(H,\nabla)$ be the graded Higgs bundle. We show the following
\begin{proposition}\label{global iso}
The relative Frobenius $\Phi$ induces an isomorphism of flat sheaves:
$$\tilde{\Phi}: C^{-1}(E,-\theta)\cong (H,\nabla).$$
\end{proposition}
\begin{remark}
In our recent work \cite{LSZ2013}, the inverse Cartier transform has been lifted over an arbitrary truncated Witt ring $W_n$ and the analogous statement holds over $W_n$ for $n\in \N$.
\end{remark}
\begin{proof}
We use the reinterpretation of $C^{-1}$ via the exponential twisting and write $(H_{\exp},\nabla_{\exp})$ for $C_{\exp}^{-1}(E,\theta)$. Take a small open affine covering $\{\mathbf{U}_{\alpha}\}$ of $\bf{X}$, together with a Frobenius lifting $F_{\bf{U}_\alpha}$ over $\hat{\bf{U}}_{\alpha}$, the $p$-adic completion of $\bf{U}_{\alpha}$. As before, we denote by $\tilde{U}_{\alpha}$ and $\tilde{F}_{\alpha}$ their mod $p^2$ reductions. Recall that over the pair ($\bf{U}_{\alpha}$, $F_{\bf{U}_{\alpha}}$),  the strong $p$-divisible property of $\Phi$ provides an local isomorphism:
$$
 \tilde{\Phi}_{\tilde{F}_\alpha}:=\sum_{i}\Phi/p^i: H_{\exp}|_{U_\alpha}=F_0^*(E|_{U_\alpha})\cong H|_{U_\alpha}.
$$
Then we shall show that the set of local isomorphisms $\{ \tilde{\Phi}_{\tilde{F}_\alpha}\}$ glues into an isomorphism $\tilde{\Phi}$ as claimed.\\

{\itshape Step 1: Sheaf isomorphism.}
 For any local section $e$ of $E$ over $U_{\alpha\beta}$, we are going to show that over $ U_{\alpha\beta}$,
$$\tilde{\Phi}_{\tilde{F}_{\alpha}}(F_0^*e)=\tilde{\Phi}_{\tilde{F}_{\beta}}\circ J_{\alpha\beta}(F_0^*e).$$
We take a system of \'{e}tale local coordinates $\{t_1,\cdots, t_d\}$ of $U_{\alpha\beta}$. For a multi-index $\underline{j}=(j_1,\cdots,j_d)$, we put  $$\theta_{\partial}^{\underline{j}}=(\partial_{t_1}\lrcorner \theta)^{j_1}\cdots (\partial_{t_d}\lrcorner \theta)^{j_d},\ z_l=\left(\frac{ \tilde{F}_{\alpha}-\tilde{F}_{\beta}}{[p]} \right)(F_0^*t_l), \ z^{\underline{j}}=\prod_{l=1}^d z_l^{j_l}.$$
As $\Phi$ is horizontal under $\nabla$, according to  the Taylor formula,  we have
$$
\tilde{\Phi}_{\tilde{F}_{\alpha}}(F_0^*e )= \tilde{\Phi}_{\tilde{F}_{\beta}} \circ \left(1+\sum_{|\underline{j}|=1}^{n}F_0^*(\theta_{\partial}^{\underline{j}})\cdot\frac{z^{\underline{j}}}{\underline{j}\,!}\right)(F_0^*e).
$$
So it suffices to show $J_{\alpha\beta}=1+\sum_{|\underline{j}|=1}^{n}F_0^*(\theta_{\partial}^{\underline{j}})\cdot\frac{z^{\underline{j}}}{ \underline{j}\,!}$. As $$
h_{\alpha\beta}(F_0^*\theta)=\sum_{l=1}^d F_0^*(\partial_{t_l}\lrcorner \theta) h_{\alpha\beta}(F_0^*dt_l),$$
and $$h_{\alpha\beta}(F_0^*dt_l)=\left(\frac{ \tilde{F}_{\alpha}-\tilde{F}_{\beta}}{[p]} \right)(F_0^*t_l)=z_l,$$
it follows that
$$
\frac{(h_{\alpha\beta}(F_0^*\theta ))^i}{i!}=\frac{(\sum_{l=1}^d F_0^*(\partial_{t_l}\lrcorner \theta)z_l)^i}{i!}=\sum_{|\underline{j}|=i}\frac{F_0^*(\theta_{\partial}^{\underline{j}})z^{\underline{j}} }{\underline{j}\, !}.
$$
Recall that $J_{\alpha\beta}=\sum_{i=0}^{n}\frac{(h_{\alpha\beta}(F_0^*\theta ))^i}{i!}$, it follows that
$$J_{\alpha\beta}=1+\sum_{|\underline{j}|=1}^{n}F_0^*(\theta_{\partial}^{\underline{j}})\cdot\frac{z^{\underline{j}}}{ \underline{j}\,!}$$
as wanted.\\

{\itshape Step 2: Connection isomorphism.} We need to show that under the above isomorphism, the connection $\nabla_{\exp}$ on $H_{\exp}$ is equal to the connection $\nabla$ on $H$. Take a local section $e$ of $E$ over $U_{\alpha}$. Set $\theta_{l}:=\partial t_l\lrcorner\theta$. By the horizontal property of $\Phi$, we have
$$
\nabla[\tilde{\Phi}_{\tilde{F}_{\alpha}}(F_0^*e)]=\tilde{\Phi}_{\tilde{F}_\alpha}\circ \left[\sum_{l=1}^d F_0^*\theta_l\cdot \zeta_{\alpha}(F_0^*dt_l)\right](F_0^*e).
$$
As
$$\nabla_{exp}(F_0^*e)=\left[\sum_{l=1}^d F_0^*\theta_l\cdot\zeta_{\alpha}(F_0^*dt_l)\right](F_0^*e),$$
it follows that
$$\tilde{\Phi}_{\tilde{F}_\alpha}(\nabla_{exp}(F_0^*e))=\left[\sum_{l=1}^d F_0^* \theta_l\cdot\zeta_{\alpha}(F_0^*dt_l)\right]\tilde{\Phi}_{\tilde{F}_\alpha}(F_0^*e) =\nabla[\tilde{\Phi}_{\tilde{F}_\alpha}(F_0^*e)].
$$
\end{proof}

{\bf Acknowledgement:} We would like to thank heartily Arthur Ogus for his very helpful comments on the old version \cite{LSZ} of this paper, especially for pointing to us that the equivalence between $C^{-1}$ and $C^{-1}_{\exp}$ up to sign should follow from Remark 2.10 \cite{OV}.

\end{document}